\documentclass[11pt,a4paper]{article}

\usepackage{cite}

\usepackage{cmap}
\usepackage[T1]{fontenc}

\usepackage{amsmath,amsthm,amssymb,latexsym,graphicx,mathrsfs}
\usepackage{secdot}
\sectiondot{subsection}
\usepackage{a4wide}
\usepackage{enumitem}

\usepackage{xcolor}
\usepackage{url}
\usepackage{hyperref}
\definecolor{ForestGreen}{rgb}{0.1,0.6,0.05}
\definecolor{EgyptBlue}{rgb}{0.063,0.1,0.6}
\definecolor{RipeOlive}{HTML}{556B2F}
\hypersetup{
	colorlinks=true,
	linkcolor=EgyptBlue,         
	citecolor=ForestGreen,
	urlcolor=RipeOlive
}

\usepackage[hyperpageref]{backref}
\usepackage{epstopdf}
\newtheorem{theorem}{Theorem}
\newtheorem{proposition}[theorem]{Proposition}
\newtheorem{lemma}[theorem]{Lemma}

\theoremstyle{definition}

\newtheorem{remark}[theorem]{Remark}

\numberwithin{equation}{section}
\numberwithin{theorem}{section}

\numberwithin{equation}{section}
\numberwithin{theorem}{section}

\newcommand{\W}{W_0^{1,p}(\Omega)} 
\newcommand{\intO}{\int_\Omega}

\title{
	\vspace*{-1cm}
	Generalized Picone inequalities and their applications to $(p,q)$-Laplace equations} 
\author{ 
	\normalsize Vladimir Bobkov\\ 
	{\small  Department of Mathematics and NTIS, Faculty of Applied Sciences, University of West Bohemia}\\ 
	{\small Univerzitn\'i 8, 301 00 Plze\v{n}, Czech Republic}\\
	{\small  Institute of Mathematics, Ufa Federal Research Centre, RAS}\\ 
	{\small Chernyshevsky str.\ 112, 450008 Ufa, Russia}\\
	{\small e-mail: bobkov@kma.zcu.cz}\\[0.5em] 	
	\normalsize Mieko Tanaka\\
	{\small Department of  Mathematics, 
		Tokyo University of Science}\\
	{\small Kagurazaka 1-3, Shinjyuku-ku, Tokyo 162-8601, Japan}\\
	{\small e-mail: miekotanaka@rs.tus.ac.jp} 
}

\date{}

\begin{document}
\maketitle 
	
\begin{abstract} 
	
	We obtain a generalization of the Picone inequality which, in combination with the classical Picone inequality, appears to be useful for problems with the $(p,q)$-Laplace type operators. 
	With its help, as well as with the help of several other known generalized Picone inequalities, we provide some nontrivial facts on the existence and nonexistence of positive solutions to the zero Dirichlet problem for the equation $-\Delta_p u -\Delta_q u = f_\mu(x,u,\nabla u)$ in a bounded domain $\Omega \subset \mathbb{R}^N$ under certain assumptions on the nonlinearity and with a special attention to the resonance case $f_\mu(x,u,\nabla u) = \lambda_1(p) |u|^{p-2} u + \mu |u|^{q-2} u$, where $\lambda_1(p)$ is the first eigenvalue of the $p$-Laplacian. 	
				
	
\end{abstract}

\section{Picone inequalities} 
Throughout this section, we denote by $\Omega$ a nonempty connected open set in $\mathbb{R}^N$, $N \geq 1$. 
The nowadays classical version of the \textit{Picone inequality} (also commonly referred to as the \textit{Picone identity}) for the $p$-Laplacian can be stated as follows.
\begin{theorem}[\!\protect{\cite[Theorem 1.1]{Alleg}}]\label{thm:picone}
	Let $p>1$ and let $u, v$ be differentiable functions in $\Omega$ such that $u > 0$, $v \geq 0$.
	Then 
	\begin{equation}\label{eq:picone0}
	|\nabla u|^{p-2} \nabla u \nabla \left(\frac{v^{p}}{u^{p-1}}\right) 
	\leq
	|\nabla v|^{p}.
	\end{equation}
	Moreover, the equality in \eqref{eq:picone0} is satisfied in $\Omega$ if and only if 
	$u \equiv kv$ for some constant $k > 0$.
\end{theorem}

In the linear case $p=2$, inequality \eqref{eq:picone0} is a direct consequence of the simple identity 
\begin{equation}\label{eq:picone-original}
\nabla u \nabla \left(\frac{v^{2}}{u}\right) = |\nabla v|^2 - \left|\nabla v - \frac{v}{u}\nabla u\right|^2
\end{equation}
whose one-dimensional version was used by \textsc{M.~Picone} in \cite[Section 2]{picone} to prove the Sturm comparison theorem. 
Subsequently, due to the nontrivial and convenient choice of the test function $\frac{v^{p}}{u^{p-1}}$, identity \eqref{eq:picone-original} and inequality \eqref{eq:picone0} appeared to be effective in the study of many other properties of various ordinary and partial differential equations and systems of both linear and nonlinear nature. 
In particular, one can mention the uniqueness and nonexistence of positive solutions, Hardy type inequalities, bounds on eigenvalues, Morse index estimates, etc. 
Such a wide range of applications particularly motivated a search of reasonable generalizations of the Picone inequality, see, e.g., the works \cite{AGW,bal,BobkovTanaka2015,Bognar,BF,ilyas,Jaros,tir,Tyagi}, although this list is far from being comprehensive.

On the other hand, during the last few decades, there has been growing interest in the investigation of various composite type operators such as the sum of the $p$- and $q$-Laplacians with $p \neq q$, the so-called \textit{$(p,q)$-Laplacian}. 
The motivation for corresponding studies comes from both the intrinsic mathematical interest and applications in natural sciences, see, for instance, \cite{agudelo,BDP,BCM,benci,CI,FMM,marcomasconi,yang} and references therein, to mention a few.
Clearly, most of the properties indicated above can be posed for problems with such operators, too. 
It is then natural to ask which generalizations of the Picone inequality are favourable to be applied to the $(p,q)$-Laplacian. 
If one tries to use $\frac{v^p}{u^{p-1}}$ or $\frac{v^q}{u^{q-1}}$ as a test function, then, taking into account \eqref{eq:picone0}, the following two quantities have to be estimated:
$$
|\nabla u|^{p-2} \nabla u \nabla \left(\frac{v^{q}}{u^{q-1}}\right) 
\quad\text{and}\quad
|\nabla u|^{q-2} \nabla u \nabla \left(\frac{v^{p}}{u^{p-1}}\right).
$$
There are at least two known generalized Picone inequalities in this regard.
The first one was obtained in \cite{BF}, where its equivalence to two convexity principles for variational integrals is also shown. 
Its particular form can be stated as follows.
\begin{theorem}[\!\protect{\cite[Proposition 2.9 and Remark 2.10]{BF}}]\label{thm:B}
	Let $p,q>1$ and let $u, v$ be differentiable functions in $\Omega$ such that $u > 0$, $v \geq 0$.
	If $q \leq p$, then 
	\begin{equation}\label{eq:piconeB}
	|\nabla u|^{p-2} \nabla u \nabla \left(\frac{v^{q}}{u^{q-1}}\right) 
	\leq
	\frac{q}{p}|\nabla v|^{p} + \frac{p-q}{p}|\nabla u|^{p}.
	\end{equation}
\end{theorem}

\begin{remark}
	Although the case of equality in \eqref{eq:piconeB} is not discussed in \cite{BF}, one can show that if $q<p$, then the equality in \eqref{eq:piconeB} is satisfied in $\Omega$ if and only if $u \equiv v$.
\end{remark}

The second generalization of \eqref{eq:picone0} was obtained in \cite{ilyas} in the context of study of an equation with indefinite nonlinearity. 
Later, this result was also applied in \cite{BobkovTanaka2015} to an eigenvalue problem for the $(p,q)$-Laplacian. 
\begin{theorem}[\!\protect{\cite[Lemma 1]{ilyas}}]\label{thm:Il}
	Let $p,q>1$ and let $u, v$ be differentiable functions in $\Omega$ such that $u > 0$, $v \geq 0$.
	If $p \leq q$, then 
	\begin{equation}\label{eq:piconeIl}
	|\nabla u|^{p-2} \nabla u \nabla \left(\frac{v^{q}}{u^{q-1}}\right) 
	\leq
	|\nabla v|^{p-2} \nabla v \nabla \left(\frac{v^{q-p+1}}{u^{q-p}}\right). 
	\end{equation}
	Moreover, the equality in \eqref{eq:piconeIl} is satisfied in $\Omega$ if and only if 
	$u \equiv kv$ for some constant $k > 0$.
\end{theorem}

\begin{remark}
	For convenience of further applications of \eqref{eq:piconeIl}, we rewrite it, assuming $q \leq p$, as follows:
	\begin{equation}\label{eq:piconeIlx}
	|\nabla u|^{q-2} \nabla u \nabla \left(\frac{v^{p}}{u^{p-1}}\right) 
	\leq
	|\nabla v|^{q-2} \nabla v \nabla \left(\frac{v^{p-q+1}}{u^{p-q}}\right). 
	\end{equation}
\end{remark}

Notice that both \eqref{eq:piconeB} and \eqref{eq:piconeIl} turn to the Picone inequality  \eqref{eq:picone0} when $p=q$. 
Moreover, we emphasize that \eqref{eq:piconeB} requires $q \leq p$, while \eqref{eq:piconeIl} asks for $p \leq q$. 
Our main result, Theorem~\ref{thm:picone-general} below, posits the fact that inequality \eqref{eq:piconeIl} remains valid for some $p>q$, although the set of feasible values of $p$ and $q$ is not of a trivial structure, see Figure \ref{fig1}. 
This set is defined and characterized in the following lemma. 
\begin{lemma}\label{lem:pic}
	Let $q>1$ be fixed. Let the function $g: [0,+\infty) \times (1,+\infty) \to \mathbb{R}$ be defined as
	$$
	g(s;p) = (q-1) s^p + q s^{p-1} - (p-q) s + (q-p+1),
	$$
	let
	\begin{equation*}\label{def:I(q)}  
	I(q) := \{p> 1:~ g(s;p)\ge 0 ~\text{for all}~ s\geq 0\},
	\end{equation*}
	and set $\widetilde{p} = \widetilde{p}(q) := \sup\{p>1: p\in I(q)\}$.
	Then $\max\{2,q\}<\widetilde{p}<q+1$. 
	Moreover, there exists $\widetilde{q} \in (1,2]$ such that the following assertions hold:
	\begin{enumerate}[label={\rm(\roman*)}]
		\item\label{lem:pic:1} if $q < \widetilde{q}$, then there exist $p_* = p_*(q)$ and $p^*=p^*(q)$ satisfying $q < p_* < p^* < 2$ such that $(1,p_*] \cup [p^*,\widetilde{p}] \subset I(q)$ and $(p_*,p^*) \not\subset I(q)$;
		\item\label{lem:pic:3} if $q \geq \widetilde{q}$, then $I(q) = (1,\widetilde{p}]$.
	\end{enumerate}
	Furthermore, if $1<q_1 < q_2$ and $p \in I(q_1)$, then $p \in I(q_2)$, i.e., $I(q_1) \subset I(q_2)$.
	
	In particular, each of the following two explicit assumptions is sufficient to guarantee that $p \in I(q)$:
	\begin{enumerate}[label={\rm(\Roman*)}]
		\item\label{rem:pic:1} $1 < q < p \leq 2$ and $p\le q+q^{p-1}(q-1)^{2-p}$; 
		\item\label{rem:pic:2} $2 \leq p < q+1$ and $(q+1-p)^{p-2} q \geq (p-q)^{p-1}$. 
	\end{enumerate}
\end{lemma}

\begin{figure}[h!]
	\center{\includegraphics[width=0.5\linewidth]{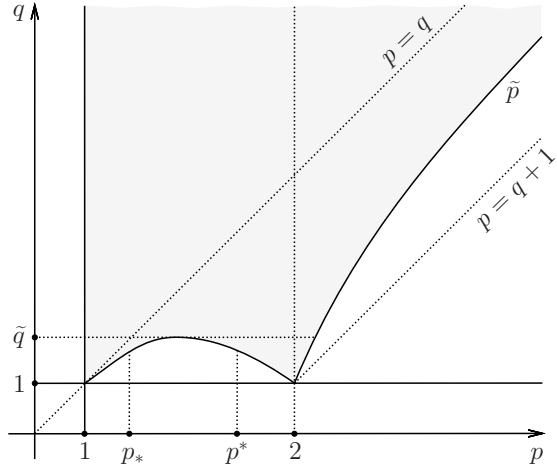}}
	\caption{The grey set schematically depicts the set of points $(p,q)$ with $p \in I(q)$.}
	\label{fig1}
\end{figure}

\begin{remark}
	The assertions \ref{lem:pic:1} and \ref{lem:pic:3} of Lemma \ref{lem:pic} yield $2 \in I(q)$ for any $q>1$. 
	A numerical investigation of the function $g$ indicates that $\widetilde{q} = 1.051633991...$ and that $p_*, p^*$ in the assertion \ref{lem:pic:1} can be chosen such that $(p_*, p^*) \cap I(q) = \emptyset$, that is, if $q < \widetilde{q}$, then $I(q) = (1,p_*] \cup [p^*,\widetilde{p}]$, see Figure \ref{fig1}.
\end{remark}

Now we are ready to state our main result.
\begin{theorem}\label{thm:picone-general}
	Let $p,q>1$ and let $u, v$ be differentiable functions in $\Omega$ such that $u > 0$, $v >0$. 
	Assume that one of the following assumptions is satisfied:
	\begin{enumerate}[label={\rm(\roman*)}]
		\item\label{thm:picone:1} $p \in I(q)$, where $I(q)$ is given by Lemma \ref{lem:pic};
		\item\label{thm:picone:2} $p \leq q+1$ and $\nabla u \nabla v \geq 0$.
	\end{enumerate}
	Then 
	\begin{equation}\label{eq:picone}
	|\nabla u|^{p-2} \nabla u \nabla \left(\frac{v^{q}}{u^{q-1}}\right) 
	\leq
	|\nabla v|^{p-2} \nabla v \nabla \left(\frac{v^{q-p+1}}{u^{q-p}}\right). 
	\end{equation}
	Moreover, if $p < q+1$ and $\nabla u \nabla v \geq 0$, then the equality in \eqref{eq:picone} is satisfied in $\Omega$ if and only if $u \equiv kv$ for some constant $k > 0$.
	
	Furthermore, the assumptions \ref{thm:picone:1} and \ref{thm:picone:2} are optimal in the following sense:
	\begin{enumerate}[label={\rm(\Roman*)}]
		\item\label{thm:picone:1x} if $p \not\in I(q)$, then there exist $u,v$ and a point $x \in \Omega$ such that \eqref{eq:picone} is violated at $x$;
		\item\label{thm:picone:2x} if $p > q+1$, then there exist $u,v$ with $\nabla u \nabla v \geq 0$ and a point $x \in \Omega$ such that \eqref{eq:picone} is violated at $x$.
	\end{enumerate}
\end{theorem}

A closer look at the proof of Theorem \ref{thm:picone-general} \ref{thm:picone:2} reveals that inequality \eqref{eq:picone} remains valid under the assumption \ref{thm:picone:2} also for $q=1$. In fact, even the following stronger result, which reduces to the commutativity of the scalar product in $W^{1,2}(\Omega)$ at $p=2$, can be obtained by the same method of proof.
\begin{proposition}
	Let $u, v$ be differentiable functions in $\Omega$ such that $u > 0$, $v > 0$, and $\nabla u \nabla v \geq 0$. 
	Then the following assertions hold:
	\begin{enumerate}[label={\rm(\roman*)}]
		\item if $p \in (1,2]$, then 
		\begin{equation}\label{eq:picone-rad1}
		|\nabla u|^{p-2} \nabla u \nabla v 
		\leq
		|\nabla v|^{p-2} \nabla v \nabla \left(u^{p-1}v^{2-p}\right);
		\end{equation}
		\item if $p\ge 2$, then 
		\begin{equation}\label{eq:picone-rad2}
		|\nabla u|^{p-2} \nabla u \nabla v 
		\geq 
		|\nabla v|^{p-2} \nabla v \nabla \left(u^{p-1}v^{2-p}\right). 
		\end{equation}
	\end{enumerate}
	Moreover, if $p \neq 2$, then the equality in \eqref{eq:picone-rad1} or \eqref{eq:picone-rad2} is satisfied in $\Omega$ if and only if $u \equiv kv$ for some constant $k > 0$.
\end{proposition}

\bigskip
Apart from the choice of $\frac{v^p}{u^{p-1}}$ or $\frac{v^q}{u^{q-1}}$ as a test function, one could also consider more general test functions of the form $\frac{v^p}{f(u)}$ or $\frac{v^q}{f(u)}$.
In this direction, the following partial case of a generalized Picone inequality obtained in \cite{tir} by applying an inequality from \cite[Lemma 2.1]{Jaros} can be effectively used.
\begin{theorem}[\!\protect{\cite[Theorem 2.2]{tir}}]\label{thm:Tir}
	Let $p>1$ and let $u, v$ be differentiable functions in $\Omega$ such that $u > 0$, $v \geq 0$.
	Assume that	$f \in C^1(0,+\infty)$ satisfies $f(s), f'(s)>0$ for all $s \in (0,+\infty)$. 
	Then
	\begin{equation}\label{eq:piconeTir}
	|\nabla u|^{p-2} \nabla u \nabla \left(\frac{v^{p}}{f(u)}\right) 
	\leq
	(p-1)^{p-1}\frac{f(u)^{p-2}}{f'(u)^{p-1}} |\nabla v|^p.
	\end{equation}
	Moreover, the equality in \eqref{eq:piconeTir} is satisfied in $\Omega$ if and only if 
	$f(u) \equiv k v^{p-1}$ for some constant $k > 0$.	
\end{theorem}

\begin{remark}
	Let $q>1$.
	Since $v^p = (v^{p/q})^q$, we get from \eqref{eq:piconeTir} the complementary inequality
	\begin{equation}\label{eq:piconeTir2}
	|\nabla u|^{q-2} \nabla u \nabla \left(\frac{v^{p}}{f(u)}\right) 
	\leq
	(q-1)^{q-1}\frac{f(u)^{q-2}}{f'(u)^{q-1}} |\nabla (v^{p/q})|^q.
	\end{equation}
	Notice that the term $|\nabla (v^{p/q})|$ is well-defined if either $q \leq p$ and $v \geq 0$, or $q \neq p$ and $v>0$.
	In particular, under any of these assumptions, taking $f(s)=s^{p-1}$, we obtain
	\begin{equation}\label{eq:piconeTir3}
	|\nabla u|^{q-2} \nabla u \nabla \left(\frac{v^{p}}{u^{p-1}}\right) 
	\leq
	\left(\frac{q-1}{p-1}\right)^{q-1} \left(\frac{p}{q}\right)^q \left(\frac{v}{u}\right)^{p-q} |\nabla v|^q.
	\end{equation}
	Moreover, under the assumption $q<p$, the application of Young's inequality gives
	\begin{equation}\label{eq:piconeTir4}
	|\nabla u|^{q-2} \nabla u \nabla \left(\frac{v^{p}}{u^{p-1}}\right) 
	\leq
	\left(\frac{q-1}{p-1}\right)^{q-1} \left(\frac{p}{q}\right)^q
	\left(\frac{p-q}{p}\left(\frac{v}{u}\right)^{p} + \frac{q}{p} |\nabla v|^p\right).
	\end{equation}
	Evidently, \eqref{eq:piconeTir3} and \eqref{eq:piconeTir4} reduce to the Picone inequality \eqref{eq:picone0} if $q=p$.
\end{remark}

As a complementary fact, we provide the following optimal refinement of a generalized Picone inequality obtained in \cite[Proposition 8]{BobkovTanaka2015}, by analysing the right-hand sides of inequalities \eqref{eq:piconeTir} and \eqref{eq:piconeTir2}. 
\begin{proposition}\label{thm:PiconeBM}
Let $p,q>1$, $\alpha,\beta > 0$, and let $u, v$ be differentiable functions in $\Omega$ such that $u > 0$, $v \geq 0$.
If $q < p$, then 
$$
|\nabla u|^{p-2} \nabla u \nabla \left(\frac{v^p}{\alpha u^{p-1} + \beta u^{q-1}}\right) 
\leq
\frac{1}{\alpha C} |\nabla v|^p 
$$
and 
$$
|\nabla u|^{q-2} \nabla u \nabla \left(\frac{v^p}{\alpha u^{p-1} + \beta u^{q-1}}\right) 
\leq
\frac{1}{\beta}|\nabla (v^{p/q})|^q,
$$
where $C = 1$ if $p \leq q+1$, and $C= \frac{(q-1)^{p-2} (p-q)}{(p-2)^{p-2}}$ if $p \geq q+1$.
\end{proposition}

\medskip
Finally, let us note that the Picone inequality \eqref{eq:picone0} can be used to derive the \textit{D\'iaz-Saa inequality} \cite{diazsaa}:
\begin{equation}\label{eq:diazsaa}
\int_\Omega \left(-\frac{\Delta_p w_1}{w_1^{p-1}}+\frac{\Delta_p w_2}{w_2^{p-1}}\right)\left(w_1^p-w_2^p\right)dx \geq 0,
\end{equation}
which, in particular, holds (in the sense of distributions) for all $w_1,w_2 \in W_0^{1,p}(\Omega)$ such that $w_i > 0$ a.e.\ in $\Omega$, and $\frac{w_1}{w_2},\frac{w_2}{w_1} \in L^\infty(\Omega)$, assuming that $\Omega$ is smooth and bounded, see \cite[Theorem 2.5]{TG}.
Inequality \eqref{eq:diazsaa} appeared to be a useful tool in the study of uniqueness of positive solutions to boundary value problems with the $p$-Laplacian.
Its generalization to the $(p,q)$-Laplacian, together with the corresponding applications, was obtained in \cite{FMM}, see also \cite{AGW,TG}.
Under the same assumptions on $w_1,w_2$ and $\Omega$ as above, it can be stated as follows. 
If $1<q<p$ and $\mu>0$, then 
\begin{equation}\label{eq:diazsaa-pq}
\int_\Omega \left(-\frac{\Delta_p w_1+\mu \Delta_q w_1}{w_1^{q-1}}+\frac{\Delta_p w_2 + \mu \Delta_q w_2}{w_2^{q-1}}\right)\left(w_1^q-w_2^q\right)dx \geq 0.
\end{equation}
Inequality \eqref{eq:diazsaa-pq} can be established by applying the generalized Picone inequality \eqref{eq:piconeB}.

\medskip
The rest of the article is organized as follows. In Section \ref{sec:proofs}, we prove Theorem \ref{thm:picone-general} and Lemma \ref{lem:pic}. In Section \ref{sec:applications}, we provide several applications of Theorem \ref{thm:picone-general}, as well as of Theorems \ref{thm:B} and \ref{thm:Il}, to problems with the $(p,q)$-Laplacian.

\section{Proofs of Theorem \ref{thm:picone-general} and Lemma \ref{lem:pic}}\label{sec:proofs}
\begin{proof}[Proof of Theorem \ref{thm:picone-general}]
	Since the case $p \leq q$ is covered by Theorem \ref{thm:Il}, we will assume hereinafter that $p > q$.
	Moreover, under any of the assumptions \ref{thm:picone:1} and \ref{thm:picone:2}, $p$ has the upper bound $p \leq q+1$ (see Lemma \ref{lem:pic} in the case of the assumption \ref{thm:picone:1}).
	
	By straightforward calculations we get
	\begin{equation}\label{eq:pic0}
	|\nabla u|^{p-2} \nabla u \nabla \left(\frac{v^{q}}{u^{q-1}}\right)
	= 
	q |\nabla u|^{p-2} \nabla u \nabla v \left(\frac{v}{u}\right)^{q-1} 
	- 
	(q-1) |\nabla u|^{p} \left(\frac{v}{u}\right)^{q}
	\end{equation}
	and
	\begin{equation}\label{eq:pic1}
	|\nabla v|^{p-2} \nabla v \nabla \left(\frac{v^{q-p+1}}{u^{q-p}}\right)
	= 
	(q-p+1) |\nabla v|^{p} \left(\frac{v}{u}\right)^{q-p} 
	+ 
	(p-q) |\nabla v|^{p-2} \nabla v \nabla u \left(\frac{v}{u}\right)^{q-p+1}.
	\end{equation} 	
	We see from \eqref{eq:pic0} and \eqref{eq:pic1} that the desired inequality \eqref{eq:picone} is equivalent to 
	\begin{align}
	\notag
	q |\nabla u|^{p-2} \nabla u \nabla v \left(\frac{v}{u}\right)^{q-1} 
	&-
	(p-q) |\nabla v|^{p-2} \nabla v \nabla u \left(\frac{v}{u}\right)^{q-p+1}
	\\
	\label{eq:pic2}
	&\leq
	(q-1) |\nabla u|^{p} \left(\frac{v}{u}\right)^{q}
	+
	(q-p+1) |\nabla v|^{p} \left(\frac{v}{u}\right)^{q-p}.
	\end{align}
	Dividing by $v^q u^{p-q}$, we reduce \eqref{eq:pic2} to 
	\begin{align}
	\notag
	\frac{\nabla u \nabla v}{uv} \left(q \left(\frac{|\nabla u|}{u}\right)^{p-2} 
	-
	(p-q) \left(\frac{|\nabla v|}{v}\right)^{p-2} \right)
	\\
	\label{eq:picone1}
	\leq
	(q-1) \left(\frac{|\nabla u|}{u}\right)^{p}
	+
	(q-p+1) \left(\frac{|\nabla v|}{v}\right)^{p}.
	\end{align}
	Recalling that $q - p + 1 \geq 0$, we see that \eqref{eq:picone1} is satisfied if its left-hand side is nonpositive. 
	Therefore, let us assume that the left-hand side of \eqref{eq:picone1} is positive. 
	In particular, we have $\nabla u \nabla v \neq 0$, and hence $|\nabla u|, |\nabla v| > 0$.
	We consider two separate cases.

	1) Suppose that
	\begin{equation*}\label{eq:picone-rec1}
	\nabla u \nabla v > 0
	\quad \text{and} \quad 
	q \left(\frac{|\nabla u|}{u}\right)^{p-2} 
	-
	(p-q) \left(\frac{|\nabla v|}{v}\right)^{p-2} > 0.
	\end{equation*}
	In this case, in order to validate \eqref{eq:picone1} it is sufficient to prove that
	\begin{align}
	\notag
	\frac{|\nabla u| |\nabla v|}{uv} \left(q \left(\frac{|\nabla u|}{u}\right)^{p-2} 
	-
	(p-q) \left(\frac{|\nabla v|}{v}\right)^{p-2} \right)
	\\
	\label{eq:pic3}
	\leq
	(q-1) \left(\frac{|\nabla u|}{u}\right)^{p}
	+
	(q-p+1) \left(\frac{|\nabla v|}{v}\right)^{p}.
	\end{align}
	Denoting $s = \frac{|\nabla u|}{u} \frac{v}{|\nabla v|}$, we see that \eqref{eq:pic3} holds provided 
	\begin{equation}\label{eq:pic:f}
	f(s) := (q-1) s^p - q s^{p-1} + (p-q) s + (q-p+1) \geq 0 \quad \text{for all}~ s \geq 0.
	\end{equation}
	Let us show that \eqref{eq:pic:f} is satisfied.
	We have
	$$
	f''(s) = p(p-1)(q-1)s^{p-2} - q(p-1)(p-2)s^{p-3} > 0 
	~~\text{if and only if}~ 
	s > \max\left\{0,\frac{q(p-2)}{p(q-1)}\right\}.
	$$
	Combining this strict convexity of $f$ with the facts that $f(1)=f'(1)=0$ and $\frac{q(p-2)}{p(q-1)} < 1$, we see that 
	$$
	f(s) \geq 0 
	\quad \text{for all}~ s \geq \max\left\{0,\frac{q(p-2)}{p(q-1)}\right\},
	$$
	and the equality $f(s)=0$ for such $s$ happens if and only if $s=1$. 
	In particular, $f(s) \geq 0$ for all $s \geq 0$ provided $p \leq 2$.
	Assume that $p > 2$.
	Since $f$ is concave on $\left[0,\frac{q(p-2)}{p(q-1)}\right]$, $f\left(\frac{q(p-2)}{p(q-1)}\right)>0$, $f(0)>0$ for $p<q+1$, and $f(0)=0$ for $p=q+1$, we conclude that 
	$$
	f(s) \geq 0 
	\quad \text{for all}~ s \in \left[0,\frac{q(p-2)}{p(q-1)}\right],
	$$
	and the equality $f(s)=0$ for such $s$ happens if and only if $s=0$ and $p=q+1$. 
	Thus, we have derived that $f(s) \geq 0$ for all $s \geq 0$ provided $p \leq q+1$. 
	In particular, this implies that \eqref{eq:picone} is satisfied under the assumption \ref{thm:picone:2}.
	Moreover, we have shown that if $p<q+1$, then $f(s)=0$ if and only if $s=1$. 
	Therefore, if $p<q+1$, $\nabla u \nabla v \geq 0$, and the equality in \eqref{eq:picone} is satisfied in $\Omega$, then we conclude that $\nabla u \nabla v = |\nabla u||\nabla v|$ and $\frac{|\nabla u|}{u} = \frac{|\nabla v|}{v}$, which yields $\nabla\left(\frac{v}{u}\right)=0$ in $\Omega$, that is, $u\equiv kv$ for some constant $k>0$.

	2) Suppose that 
	\begin{equation}\label{eq:picone-rec2}
	\nabla u \nabla v < 0
	\quad \text{and} \quad 
	q \left(\frac{|\nabla u|}{u}\right)^{p-2} 
	-
	(p-q) \left(\frac{|\nabla v|}{v}\right)^{p-2} < 0.
	\end{equation}
	To establish \eqref{eq:picone1} under the assumption \eqref{eq:picone-rec2}, it is sufficient to show that
	\begin{align}
	\notag
	\frac{|\nabla u| |\nabla v|}{uv} \left(-q \left(\frac{|\nabla u|}{u}\right)^{p-2} 
	+
	(p-q) \left(\frac{|\nabla v|}{v}\right)^{p-2} \right)
	\\
	\label{eq:pic4}
	\leq
	(q-1) \left(\frac{|\nabla u|}{u}\right)^{p}
	+
	(q-p+1) \left(\frac{|\nabla v|}{v}\right)^{p}.
	\end{align}
	Introducing again the notation $s = \frac{|\nabla u|}{u} \frac{v}{|\nabla v|}$, we see that inequality \eqref{eq:pic4} holds if 
	\begin{equation}\label{eq:pic6}
	g(s;p) := (q-1) s^p + q s^{p-1} - (p-q) s + (q-p+1) \geq 0 \quad \text{for all}~ s \geq 0.
	\end{equation}
	Applying Lemma \ref{lem:pic}, we deduce that \eqref{eq:pic6} is satisfied whenever $p \in I(q)$. 
	
	Combining the cases 1) and 2), we conclude that \eqref{eq:picone} holds under the assumption \ref{thm:picone:1}, which finishes the proof of the first part of the theorem.
	
	Let us now obtain the optimality of the assumptions \ref{thm:picone:1} and \ref{thm:picone:2} stated in \ref{thm:picone:1x} and \ref{thm:picone:2x}, respectively.
	Assume first that $p \not\in I(q)$ and let $s_0 \geq 0$ be such that $g(s_0;p)<0$.
	Consider
	$$
	u(x_1,\dots,x_N) = 1 - \alpha x_1
	\quad \text{and} \quad
	v(x_1,\dots,x_N) = 1 + x_1
	$$
	for some $\alpha \geq 0$.
	Noting that $\frac{|\nabla u(0)|}{u(0)} \frac{v(0)}{|\nabla v(0)|} = \alpha$ and taking $\alpha = s_0$, we conclude that the violation of \eqref{eq:pic6} at $s_0$ implies the violation of \eqref{eq:pic4} at $x=0$.
	On the other hand, we have $\nabla u \nabla v = -|\nabla u||\nabla v|$. 
	Thus, the violation of \eqref{eq:pic4} at $x=0$ is equivalent to the violation of 
	\eqref{eq:picone1} at $x=0$, which, in its turn, is equivalent to the violation of 
	\eqref{eq:picone} at $x=0$. This establishes the case \ref{thm:picone:1x}.
	
	Assume now that $p > q+1$. Set $u \equiv \text{const} > 0$ and let $v>0$ be any differentiable function not identically equal to a constant. We readily see that $\nabla u \nabla v \equiv 0$ and \eqref{eq:pic2} is violated at points where $|\nabla v| > 0$, which establishes the case \ref{thm:picone:2x}.
\end{proof}

Now we provide the proof of Lemma \ref{lem:pic}. 
\begin{proof}[Proof of Lemma \ref{lem:pic}]	
	To prove that $\max\{2,q\} < \widetilde{p} < q+1$, we first note that
	\begin{align}
	\label{eq:lem:pic:1}
	g(s;2) &= (q-1)s^2 + 2(q-1) s + (q-1) \geq q-1 > 0,\\
	\notag
	g(s;q) &= (q-1)s^q + qs^{q-1}+1 \geq 1,
	\end{align}
	for all $s \geq 0$. 
	This yields $2,q \in I(q)$, and hence $\max\{2,q\} \leq \widetilde{p}$. 
	Second, we have $g(0;p) = q-p+1 < 0$ for any $p > q+1$, which implies that $\widetilde{p} \leq q+1$.
	Third, we see that
	\begin{equation*}\label{eq:lem:pic:3}
	g(s;q+1) = (q-1)s^{q+1} + qs^q -s  < 0
	\quad \text{for all sufficiently small}~ s > 0,	 
	\end{equation*}
	and
	\begin{equation}\label{eq:lem:pic:3x}
	g(s;p) \geq (q-1)s^p + (2q-p) s + (q-p+1) > 0 
	\quad \text{for all}~ s \geq 1 ~\text{and}~ 2 \leq p \leq q+1.
	\end{equation}
	Therefore, since $g$ is uniformly continuous on compact subsets of $[0,1] \times (1,+\infty)$, we conclude that $\max\{2,q\} < \widetilde{p} < q+1$. 
	Moreover, the continuity of $g$ gives $\widetilde{p}\in I(q)$. 
	
	Let us prove the assertions \ref{lem:pic:1} and \ref{lem:pic:3}.
	To this end, we notice that 
	\begin{equation}\label{eq:monotone}
	g(s;p_1) > g(s;p_2) \quad {\rm for\ every}\ s\in[0,1]\quad 
	{\rm provided}\ 1 < p_1 < p_2. 
	\end{equation}
	In view of this monotonicity, inequalities \eqref{eq:lem:pic:1} and \eqref{eq:lem:pic:3x}
	yield $[2,\widetilde{p}]\subset I(q)$.
	At the same time, we see that if $p \leq q$, then
	$$
	g(s;p) \geq (q-1) s^p + q s^{p-1} -(p-q)s \geq (q-1) s^p + q s^{p-1} \geq 0 \quad \text{for all}~ s \geq 0,
	$$
	which shows, in particular,  that $(1,q] \subset I(q)$. 
	Thus, we deduce that $I(q) = (1,\widetilde{p}]$ provided $q \geq 2$. 
	Let us now define the critical value
	$$
	\widetilde{q} = \inf\left\{q>1:~I(q) = (1,\widetilde{p}]\,\right\}.
	$$
	Clearly, $\widetilde{q} \leq 2$.
	Since $q$ is presented in the function $g$ only as a positive coefficient, we see that $p \in I(q_1)$ implies $p \in I(q_2)$ provided $q_1 < q_2$, i.e., $I(q_1) \subset I(q_2)$.
	Combining this fact with the already obtained inclusion $[2,\widetilde{p}]\subset I(q)$, we conclude that if $q > \widetilde{q}$, then $I(q) = (1,\widetilde{p}]$. 
	Notice that if $q$ is sufficiently close to $1$, then $I(q) \neq (1,\widetilde{p}]$, and so $\widetilde{q}>1$. 
	Indeed, choosing, for instance, $q=1.05$, we have $g(60; 1.3)=-	0.417508...$, which yields $1.3 \not\in I(q)$, and hence $\widetilde{q} \geq 1.05 >1$. 
	This implies, by the continuity of $g$, that $I(q) = (1,\widetilde{p}]$ also for $q=\widetilde{q}$, and this completes the proof of the assertion \ref{lem:pic:3}.
	
	Assume now that $1<q<\widetilde{q}$. In particular, recalling that $\widetilde{q} \leq 2$, we have $1<q<2$.
	We start by showing that, in addition to \eqref{eq:lem:pic:3x}, there exists $p^* \in (q,2)$ with the property that $g(s;p) \geq 0$ for all $s \geq 1$ and $p \in [p^*, 2]$. 
	Indeed, suppose, by contradiction, that for any $n \in \mathbb{N}$ one can find $p_n \in (q, 2)$ and $s_n \geq 1$ such that $g(s_n;p_n)<0$, and $p_n \to 2$ as $n \to +\infty$. 
	Then $\{s_n\}$ must be bounded, since otherwise $g(s_n;p_n) \to +\infty$ as $n \to +\infty$. 
	Therefore, passing to the limit along appropriate subsequences of $\{p_n\}$ and $\{s_n\}$, we get a contradiction to \eqref{eq:lem:pic:3x}.
	Thus, the monotonicity \eqref{eq:monotone} 
	in combination with \eqref{eq:lem:pic:1} and \eqref{eq:lem:pic:3x} yields $[p^*, \widetilde{p}] \subset I(q)$, which establishes the existence of $p^*$ from the assertion \ref{lem:pic:1}.
	
	Now we show the existence of $p_* \in (q,2)$ such that $g(s;p) \geq 0$ for all $s \geq 0$ and $p \in (q,p_*]$. 
	Suppose, by contradiction, that for any $n \in \mathbb{N}$ one can find $p_n \in (q,2)$ and $s_n>0$ satisfying $g(s_n;p_n)<0$, and $p_n \to q$ as $n \to +\infty$. 
	Since the term $(p_n-q)s_n$ is the only term in $g(s_n;p_n)$ with negative sign, we conclude that $s_n \to +\infty$ as $n \to +\infty$. 
	But then we deduce that
	$$
	0 > g(s_n;p_n) \geq (q-1) s_n^q + q s_n^{q-1} - (2-q) s_n + (q-1) > 0
	$$
	for all sufficiently large $n$, 
	since $(q-1) s_n^q$ is the leading term as $s_n \to +\infty$, which is impossible. 
	Recalling that $(1,q] \subset I(q)$, we conclude that $(1,p_*] \subset I(q)$.
	Since $q<\widetilde{q}$, we deduce that $q< p_*< p^* < 2$ and $(p_*,p^*) \setminus I(q) \neq \emptyset$, which completes the proof of the assertion \ref{lem:pic:1}.
	
	\bigskip
	Finally, we justify the sufficient assumptions \ref{rem:pic:1} and \ref{rem:pic:2}.
	
	\ref{rem:pic:1}
	Let $1 < q < p \leq 2$. 
	Considering the sum of the second and third terms of $g(s;p)$, we see that if $s^{2-p} \leq \frac{q}{p-q}$, then $g(s;p) \geq 0$. Thus, let us assume that $s^{2-p} > \frac{q}{p-q}$ and $p<2$. Then we have
	\begin{align*}
	g(s;p) 
	\geq 
	(q-1)s^p-(p-q)s
	&=
	s \, \frac{p-q}{q} \left( (q-1)s^{p-1}\left(\frac{q}{p-q}\right) - q\right)
	\\
	&> 
	s \, \frac{p-q}{q} \left( (q-1)\left(\frac{q}{p-q}\right)^{\frac{p-1}{2-p}+1} - q\right) \geq 0,
	\end{align*}
	where the last inequality is satisfied if and only if $p\le q+q^{p-1}(q-1)^{2-p}$.
	
	\ref{rem:pic:2}
	Let $2 \leq p < q+1$. As in the previous case, we see that if $s^{p-2} \geq \frac{p-q}{q}$, then $g(s;p) \geq 0$. Hence, we assume that $s^{p-2} < \frac{p-q}{q}$ and $p>2$.
	Then we have
	$$
	g(s;p) \geq -(p-q)s+(q-p+1) > -\frac{(p-q)^\frac{p-1}{p-2}}{q^\frac{1}{p-2}} +(q-p+1) \geq 0,
	$$
	where the last inequality is satisfied if and only if $(q+1-p)^{p-2} q \geq (p-q)^{p-1}$.
\end{proof}

\section{Applications to \texorpdfstring{$(p,q)$}{(p,q)}-Laplace equations}\label{sec:applications} 
Throughout this section, we always assume that $1<q<p$ and that $\Omega \subset \mathbb{R}^N$ is a smooth bounded domain with boundary $\partial \Omega$, $N \geq 2$. 

Denote by $\|\cdot\|_r$ the standard norm of $L^r(\Omega)$, $1 \le r \le +\infty$. 
Let $\lambda_1(r)$ with $1<r<+\infty$ stand for the first eigenvalue of the Dirichlet $r$-Laplacian in $\Omega$, and let $\varphi_r$ be the corresponding first eigenfunction which we assume to be positive and normalized as $\|\nabla \varphi_r\|_r = 1$.
That is,
$$
\lambda_1(r)=\inf\left\{\frac{\intO |\nabla u|^r\,dx}{\intO |u|^r\,dx}:~ 
u\in W_0^{1,r}(\Omega) \setminus\{0\} \right\} 
\quad {\rm and} \quad 
\lambda_1(r)\|\varphi_r\|_r^r = \|\nabla \varphi_r\|_r^r = 1. 
$$
Notice that $\lambda_1(r)$ is simple and $\varphi_r \in {\rm int}\,C_0^1(\overline{\Omega})_+$, where 
\begin{align*}
C_0^1(\overline{\Omega}) 
&:= 
\left\{
	u \in C^1(\overline{\Omega}):~
	u(x)=0 \text{ for all } x \in \partial \Omega
	\right\},\\
{\rm int}\,C_0^1(\overline{\Omega})_+ 
&:= 
\left\{
u \in C^1_0(\overline{\Omega}):~
u(x)>0 \text{ for all } x \in \Omega,~
\frac{\partial u}{\partial\nu}(x) < 0 \text{ for all } x \in \partial\Omega 
\right\},
\end{align*}
and $\nu$ is the unit exterior normal vector to $\partial \Omega$.  
Finally, for a weight function $m\in L^1(\Omega)$ satisfying $\intO m(x)\varphi_p^q\,dx>0$, 
we define 
\begin{equation}\label{def:beta} 
\beta_*^m = 
\frac{\intO |\nabla \varphi_p|^q\,dx }{\intO m(x)\varphi_p^q\,dx} 
\quad {\rm and} \quad 
\beta_*= 
\frac{\intO |\nabla \varphi_p|^q\,dx }{\intO \varphi_p^q\,dx}. 
\end{equation}
We remark that $\beta_* > \lambda_1(q)$, which follows from the simplicity of $\lambda_1(q)$ and linear independence of $\varphi_p$ and $\varphi_q$, see \cite[Proposition 13]{BobkovTanaka2017}.

\subsection{General problem with \texorpdfstring{$(p,q)$}{(p,q)}-Laplacian}
Consider the boundary value problem
\begin{equation}\label{eq:D}
\left\{
\begin{aligned}
-\Delta_p u -\Delta_q u &= f_\mu(x,u,\nabla u) 
&&\text{in } \Omega, \\
u&=0 &&\text{on } \partial \Omega,
\end{aligned}
\right.
\end{equation}
where the function $f_\mu(x,s,\xi): \Omega \times \mathbb{R} \times \mathbb{R}^N \to \mathbb{R}$ is sufficiently regular in order that \eqref{eq:D} possesses a weak formulation with respect to $W_0^{1,p}(\Omega)$, and satisfies the following assumption:
\begin{enumerate}[label={\textbf{(A)}}]
	\item\label{A} there exist $M \subseteq \mathbb{R}$ and $m\in L^1(\Omega)$ satisfying $\intO m(x)\varphi_p^q\,dx>0$ such that 
	$$
	f_\mu(x,s,\xi) > \lambda_1(p) s^{p-1} + \beta_*^m m(x) s^{q-1}
	$$ 
	for all $\mu \in M$, $s>0$, $\xi \in \mathbb{R}^N$, and a.e.\ $x \in \Omega$, where $\beta_*^m$ is given by \eqref{def:beta}. 
\end{enumerate}

We obtain the following nonexistence result in the class of ${\rm int}\,C_0^1(\overline{\Omega})_+$-solutions.
\begin{theorem}\label{thm:nonexistence}
	Let $p \in I(q)$, where $I(q)$ is defined in Lemma \ref{lem:pic}, and let \textnormal{\ref{A}} be satisfied.
	If $\mu \in M$, then \eqref{eq:D} has no solution in ${\rm int}\,C_0^1(\overline{\Omega})_+$.
\end{theorem}
\begin{proof}
	Suppose, by contradiction, that \eqref{eq:D} possesses a solution $u \in {\rm int}\,C_0^1(\overline{\Omega})_+$ for some $\mu \in M$. 
	Noting that $\frac{\varphi_p}{u} \in L^\infty(\Omega)$ since $\varphi_p, u\in {\rm int}\,C_0^1(\overline{\Omega})_+$, we choose $\frac{\varphi_p^q}{u^{q-1}}$ as a test function for \eqref{eq:D}. 
	Applying Theorems \ref{thm:picone} and \ref{thm:picone-general}, we get
	\begin{align*}
	&\lambda_1(p) \intO u^{p-q} \varphi_p^q \, dx 
	+ \beta_*^m \intO m(x)\varphi_p^q \,dx\\
	&< \intO f_\mu(x,u,\nabla u)\frac{\varphi_p^q}{u^{q-1}} \,dx
	=
	\intO |\nabla u|^{p-2} \nabla u\nabla \left(\frac{\varphi_p^q}{u^{q-1}}\right) dx 
	+ 
	\intO |\nabla u|^{q-2}\nabla u\nabla \left(\frac{\varphi_p^q}{u^{q-1}}\right) dx
	\\
	&\leq
	\intO |\nabla \varphi_p|^{p-2} \nabla \varphi_p \nabla \left(\frac{\varphi_p^{q-p+1}}{u^{q-p}}\right) dx
	+
	\intO |\nabla \varphi_p|^q \, dx
	\\
	&=
	\lambda_1(p) \intO u^{p-q} \varphi_p^q \, dx + \beta_*^m \intO m(x)\varphi_p^q \,dx,
	\end{align*}
	which is impossible. 
\end{proof}

Notice that we can not claim that each weak solution of \eqref{eq:D} belongs to ${\rm int}\,C_0^1(\overline{\Omega})_+$ since we do not impose suitable ``good'' assumptions on the function $f_\mu$ apart from those stated above; see, e.g., \cite{BP} for a related discussion. 
We also remark that neither of the generalized Picone inequalities \eqref{eq:piconeB}, \eqref{eq:piconeIl}, \eqref{eq:piconeTir} can be used (at least, as directly as \eqref{eq:picone}) to establish Theorem \ref{thm:nonexistence}.

\subsection{Eigenvalue-type problem} 
In the special case $f_\mu(x,s,\xi) = \lambda_1(p) |s|^{p-2} s + \mu |s|^{q-2}s$, \eqref{eq:D} can be seen as an eigenvalue problem for the $(p,q)$-Laplacian: 
\begin{equation}\label{eq:De}
\left\{
\begin{aligned}
-\Delta_p u -\Delta_q u &= \lambda_1(p) |u|^{p-2}u + \mu |u|^{q-2} u 
&&\text{in } \Omega, \\
u&=0 &&\text{on } \partial \Omega,
\end{aligned}
\right.
\end{equation}
see, e.g., \cite{BobkovTanaka2015,BobkovTanaka2017,CS,marcomasconi}. 
Notice that any nonzero and nonnegative solution of \eqref{eq:De} belongs to 
${\rm int}\,C_0^1(\overline{\Omega})_+$, see, for instance, \cite[Remark 1]{BobkovTanaka2015} or \cite[Section 2.4]{marcomasconi}. 
Although in the works \cite{BobkovTanaka2015,BobkovTanaka2017} by the present authors the structure of the set of positive solutions to a general version of \eqref{eq:De} with two parameters has been comprehensively studied, we were not able to characterize completely the range of values of $\mu$ for which \eqref{eq:D} possesses a positive solution. 
Thanks to our generalized Picone inequality \eqref{eq:picone}, as well as to inequalities \eqref{eq:piconeB} and \eqref{eq:piconeIl}, we can provide additional information in this regard. 

First, the same reasoning as in Theorem \ref{thm:nonexistence} allows to show the following result.
\begin{theorem}\label{thm:nonexistence0}
	Assume that one of the following assumptions is satisfied:
	\begin{enumerate}[label={\rm(\roman*)}]
		\item\label{thm:nonexistence:1} 
		$p \in I(q)$, where $I(q)$ is defined in Lemma \ref{lem:pic};
		\item\label{thm:nonexistence:2} $p \leq q+1$ and $\Omega$ is an $N$-ball.
	\end{enumerate}	
	Then \eqref{eq:De} has no positive solution for $\mu > \beta_*$, where $\beta_*$ is given by \eqref{def:beta}.
	Moreover, if $p<q+1$ and $\Omega$ is an $N$-ball, then \eqref{eq:De} has no positive solution also for $\mu = \beta_*$.
\end{theorem}
\begin{proof}
	Let $u$ be a positive solution of \eqref{eq:De}. 
	Recall that $u\in {\rm int}\,C_0^1(\overline{\Omega})_+$. 
	Moreover, under the assumption \ref{thm:nonexistence:2}, both $u$ and $\varphi_p$ are radially symmetric with respect to the centre of $\Omega$ and nonincreasing in the corresponding radial direction (see \cite[Theorem 3.10]{BrockTakac}), which yields $\nabla u \nabla \varphi_p \geq 0$ in $\Omega$.
	Clearly, $f_\mu(x,s,\xi) = \lambda_1(p) |s|^{p-2} s + \mu |s|^{q-2}s$ satisfies \textnormal{\ref{A}} with $m\equiv 1$ (so $\beta_*^m=\beta_*$) and $M=(\beta_*,+\infty)$. 
	Therefore, applying Theorem \ref{thm:picone-general} as in the proof of Theorem \ref{thm:nonexistence}, we obtain the desired nonexistence for $\mu > \beta_*$.
	On the other hand, if $\mu=\beta_*$, $p<q+1$, and $\Omega$ is an $N$-ball, then Theorem \ref{thm:picone-general} yields $u \equiv k \varphi_p$ for some $k>0$. Since $u$ is a solution of \eqref{eq:De}, we see that $u$ must be also an eigenfunction of the $q$-Laplacian associated with the eigenvalue $\beta_*$, which is impossible in view of \cite[Proposition 13]{BobkovTanaka2017}.
\end{proof}

Notice that Theorem \ref{thm:nonexistence0} is optimal for the considered range of $p$ and $q$ in the sense that for any $\mu \in (\lambda_1(q),\beta_*)$ problem \eqref{eq:De} possesses a positive solution, see \cite[Theorem 2.5 (i)]{BobkovTanaka2017}. 
Nevertheless, we are not aware of the corresponding existence result for $\mu=\beta_*$ under the assumption \ref{thm:nonexistence:1} of Theorem \ref{thm:nonexistence0}, or if $p=q+1$ and $\Omega$ is an $N$-ball.

Second, we provide the following general result without restrictions on $p$ and $q$ apart from the default assumption $1<q<p$, whose proof is based on a nontrivial application of Picone's inequalities \eqref{eq:piconeB} and \eqref{eq:piconeIl}, and on the usage of results from \cite{BobkovTanaka2015,BobkovTanaka2017}.
\begin{theorem}\label{thm:summary} 
	Let
	\begin{equation*}\label{def:mu} 
	\widetilde{\mu}
	:=
	\sup\left\{\mu\in\mathbb{R}:~ 
	\eqref{eq:De} ~\text{has a positive solution}\right\}. 
	\end{equation*}
	Then $\beta_*\le \widetilde{\mu}<+\infty$. 
	Moreover, \eqref{eq:De} has at least one positive solution if $\lambda_1(q) < \mu < \widetilde{\mu}$, 
	and no positive solution if $\mu\le \lambda_1(q)$ or $\mu>\widetilde{\mu}$. 
	Furthermore, if $\widetilde{\mu}>\beta_*$, then \eqref{eq:De} has at least one positive solution if and only if $\lambda_1(q) < \mu \leq \widetilde{\mu}$. 
\end{theorem} 

In order to prove Theorem \ref{thm:summary}, we need the following auxiliary information on the behaviour of positive solutions. 
\begin{proposition}\label{prop:behavior} 
	Let $\{\mu_n\} \subset \mathbb{R}$ be a sequence, and let $u_n$ be a positive solution of \eqref{eq:De} with $\mu= \mu_n$, $n \in \mathbb{N}$. 
	Then the following assertions hold: 
	\begin{enumerate}[label={\rm(\roman*)}]
		\item\label{prop:behaviour:1} if $\lim\limits_{n\to +\infty}\|\nabla u_n\|_p = +\infty$, 
		then $\lim\limits_{n\to +\infty}\mu_n=\beta_*$ and, up to a subsequence,
		\begin{equation}\label{eq:converg}
		v_n:=\frac{u_n}{\|\nabla u_n\|_p} \to \varphi_p
		~~\text{in}~~ C^1_0(\overline{\Omega})
		~~\text{as}~~ n\to +\infty;
		\end{equation}
		\item\label{prop:behaviour:2} if $\lim\limits_{n\to +\infty}\|\nabla u_n\|_p=0$, 
		then $\lim\limits_{n\to +\infty}\mu_n=\lambda_1(q)$.
	\end{enumerate} 
\end{proposition}
\begin{proof}
	We start with the observation that \eqref{eq:De} has no nonzero solution for $\mu \leq \lambda_1(q)$, see \cite[Proposition 1]{BobkovTanaka2015} and \cite[Proposition 13]{BobkovTanaka2017}. 
	Thus, throughout the proof, we will assume that $\mu_n > \lambda_1(q)$ for all $n \in \mathbb{N}$.
	In particular, we have $\liminf\limits_{n\to +\infty}\mu_n\ge \lambda_1(q)$. 
	
	\ref{prop:behaviour:1}
	Let $\|\nabla u_n\|_p \to +\infty$ as $n \to +\infty$.
	Note first that $\liminf\limits_{n\to +\infty}\mu_n > \lambda_1(q)$. 
	Indeed, suppose, by contradiction, that $\mu_n \to \lambda_1(q)$, up to a subsequence. 
	Setting $v_n = \frac{u_n}{\|\nabla u_n\|_p}$ and taking $u_n$ as a test function for \eqref{eq:De} with $\mu=\mu_n$, we have
	$$
	1 - \lambda_1(p)\int_\Omega v_n^p \, dx 
	=\|\nabla u_n\|_p^{q-p}
	\left(\mu_n \int_\Omega v_n^q \, dx - \int_\Omega |\nabla v_n|^q \, dx\right) \geq 0,
	$$
	where the inequality follows from the definition of $\lambda_1(p)$.
	Since $q<p$, $\|\nabla u_n\|_p \to +\infty$, and $\mu_n \to \lambda_1(q)$, we conclude that, simultaneously, $v_n \to \varphi_p$ and $v_n \to k \varphi_q$ strongly in $L^q(\Omega)$, up to a subsequence, where $k>0$ is some constant.
	However, this contradicts the linear independence of $\varphi_p$ and $\varphi_q$, see \cite[Proposition 13]{BobkovTanaka2017}, and hence $\liminf\limits_{n\to +\infty} \mu_n > \lambda_1(q)$.
	
	Now we prove the convergence \eqref{eq:converg}.
	Let $v_0 \in \W$ be such that $v_n \to v_0$ weakly in $\W$ and strongly in $L^p(\Omega)$, up to a subsequence.  
	First, we show that $v_0 \not\equiv 0$ in $\Omega$.
	Suppose, by contradiction, that $v_0 \equiv 0$ in $\Omega$. 
	Then, by Egorov's theorem, $v_n$ converges to $0$ uniformly on a subset of $\Omega$ of positive measure.
	In particular, we have
	\begin{equation}\label{eq:1conv}
	\intO v_n^{q-p} \varphi_q^p \, dx \to +\infty
	\quad \text{as}~ n\to +\infty.
	\end{equation}
	Notice that the integral in \eqref{eq:1conv} is well-defined since $\frac{\varphi_q}{v_n} \in L^\infty(\Omega)$ due to the ${\rm int}\,C_0^1(\overline{\Omega})_+$-regularity of $v_n$ and $\varphi_q$.
	Using now the Picone inequalities \eqref{eq:picone0} and \eqref{eq:piconeIlx}, we get from \eqref{eq:De} with $\mu=\mu_n$ and the test function $\frac{\varphi_q^{p}}{u_n^{p-1}} \in W_0^{1,p}(\Omega)$ that
	\begin{align}
	\notag
	\lambda_1(p) \intO \varphi_q^p \, dx 
	&+ \mu_n \intO u_n^{q-p} \varphi_q^p \, dx \\
	\notag
	&=
	\intO |\nabla u_n|^{p-2} \nabla u_n \nabla \left(\frac{\varphi_q^{p}}{u_n^{p-1}}\right)  dx
	+
	\intO |\nabla u_n|^{q-2} \nabla u_n \nabla \left(\frac{\varphi_q^{p}}{u_n^{p-1}}\right) dx \\
	\notag
	&\leq 
	\intO |\nabla \varphi_q|^p \, dx 
	+ 
	\intO |\nabla \varphi_q|^{q-2} \nabla \varphi_q \nabla \left(\frac{\varphi_q^{p-q+1}}{u_n^{p-q}}\right) dx
	\\
	\label{eq:pic1115}
	&=
	\intO |\nabla \varphi_q|^p \, dx + \lambda_1(q) \intO u_n^{q-p} \varphi_q^p \, dx.
	\end{align}
	This implies that
	\begin{equation}\label{eq:1}
	(\mu_n - \lambda_1(q)) \intO u_n^{q-p} \varphi_q^p \, dx \leq \intO |\nabla \varphi_q|^p \, dx - \lambda_1(p) \intO \varphi_q^p \,dx < +\infty,
	\end{equation}
	and hence, since $\liminf\limits_{n\to +\infty}\mu_n > \lambda_1(q)$, there exists a constant $C>0$ independent of $n$ such that 
	\begin{equation}\label{eq:2}
	\mu_n \intO v_n^{q-p} \varphi_q^p \, dx \leq C \|\nabla u_n\|_p^{p-q}.
	\end{equation}
	On the other hand, choosing $u_n$ as a test function for \eqref{eq:De} with $\mu=\mu_n$, we get
	\begin{equation*}
	\mu_n \intO v_n^q \, dx - \intO |\nabla v_n|^q \, dx = \|\nabla u_n\|_p^{p-q} \left(1- \lambda_1(p) \int_\Omega v_n^p \, dx \right).
	\end{equation*}
	Since we suppose that $v_0 \equiv 0$ in $\Omega$, we have $v_n \to 0$ strongly in $L^p(\Omega)$ and $L^q(\Omega)$, which yields
	\begin{equation}\label{eq:2conv} 
	2\mu_n \intO v_n^q \, dx \geq  \|\nabla u_n\|_p^{p-q}
	\quad \text{for sufficiently large}~ n \in \mathbb{N}.
	\end{equation}
	Combining \eqref{eq:2} and \eqref{eq:2conv}, we obtain 
	\begin{equation*}\label{eq:3conv} 
	2C \intO v_n^q \, dx \geq \intO v_n^{q-p} \varphi_q^p \, dx
	\quad \text{for sufficiently large}~ n \in \mathbb{N},
	\end{equation*}
	which gives a contradiction to \eqref{eq:1conv} and the strong convergence $v_n \to 0$ in $L^q(\Omega)$. 
	Therefore, $v_0 \not\equiv 0$ in $\Omega$. 
	
	Second, we show that $v_0 = \varphi_p$.
	Since $u_n$ is a solution of \eqref{eq:De} with $\mu=\mu_n$, we see that $v_n$ satisfies
	\begin{align}
	\notag
	&\intO |\nabla v_n|^{p-2}\nabla v_n\nabla\varphi \,dx 
	+ \frac{1}{\|\nabla u_n\|_p^{p-q}}\intO |\nabla v_n|^{q-2}\nabla v_n\nabla\varphi \,dx 
	\\\label{eq:vn-solution}
	&\qquad 
	= \lambda_1(p) \intO v_n^{p-1} \varphi \, dx + \frac{\mu_n}{\|\nabla u_n\|_p^{p-q}} \intO v_n^{q-1} \varphi\,dx
	\quad \text{for any}~ \varphi \in \W.
	\end{align} 
	Taking $\varphi = v_n$ and recalling that $\|\nabla v_n\|_p=1$ and that $v_n$ converges in $L^p(\Omega)$ to a nonzero function $v_0$, we conclude that there exists a constant $B \geq 0$ such that
	$$
	B_n:=\frac{\mu_n}{\|\nabla u_n\|_p^{p-q}} \to B
	\quad 
	\text{as}~ n \to +\infty.
	$$
	Taking now $\varphi=v_n-v_0$ in \eqref{eq:vn-solution}, we see that the boundedness of $B_n$ implies
	$$
	\lim_{n\to +\infty}\intO |\nabla v_n|^{p-2}\nabla v_n(\nabla v_n-\nabla v_0) \,dx
	=0, 
	$$
	which guarantees that $v_n\to v_0$ strongly in $\W$, see, e.g., \cite[Lemma 5.9.14]{Drabek}.
	Passing to the limit in \eqref{eq:vn-solution}, we deduce that $v_0$ is a nonzero and nonnegative solution of the problem
	\begin{equation*}
	\left\{
	\begin{aligned}
	-\Delta_p u &= \lambda_1(p) |u|^{p-2}u + B |u|^{q-2} u 
	&&\text{in } \Omega, \\
	u&=0 &&\text{on } \partial \Omega.
	\end{aligned}
	\right.
	\end{equation*}
	The standard regularity result \cite{Lieberman} and the strong maximum principle yield $v_0\in {\rm int}\,C_0^1(\overline{\Omega})_+$, and so $\frac{\varphi_p^p}{v_0^{p-1}} \in \W$. 
	Thus, applying the Picone inequality \eqref{eq:picone0}, we get
	\begin{align*}
	\lambda_1(p) \intO  \varphi_p^p \, dx + B \intO v_0^{q-p} \varphi_p^p \,dx
	&=\intO |\nabla v_0|^{p-2}\nabla v_0\nabla
	\left(\dfrac{\varphi_p^p}{v_0^{p-1}}\right) \,dx \\ 
	&\leq
	\intO |\nabla \varphi_p|^p \,dx = \lambda_1(p) \intO  \varphi_p^p \, dx, 
	\end{align*}
	which yields $B=0$, and hence $v_0 \equiv \varphi_p$ in $\Omega$. 
	
	Now we are ready to prove that $v_n \to \varphi_p$ in $C_0^1(\overline{\Omega})$.
	Thanks to the boundedness of $B_n$,  
	using the Moser iteration process in \eqref{eq:vn-solution}, 
	we can find $M_1>0$ independent of $n$ such that $\|v_n\|_{\infty} \le M_1$ 
	for all $n$. 
	Thus, since $1/\|\nabla u_n\|_p^{p-q}$ is also bounded, 
	applying to equation \eqref{eq:vn-solution} the regularity results \cite[Theorem 1.7]{L} and \cite{Lieberman}, we derive the existence of $\theta\in(0,1)$ and $M_2>0$, both independent of $n$, such that $v_n \in C_0^{1,\theta}(\overline{\Omega})$ and 
	$\|v_n\|_{C_0^{1,\theta}(\overline{\Omega})} \le M_2$ 
	for every sufficiently large $n$. Since $C_0^{1,\theta}(\overline{\Omega})$ 
	is compactly embedded into $C_0^1(\overline{\Omega})$, we conclude that
	$v_n \to \varphi_p$ in $C_0^1(\overline{\Omega})$, up to a subsequence.

	Finally, let us show that $\lim\limits_{n\to +\infty}\mu_n=\beta_*$.
	First, let $\{\mu_{n_k}\}$ be a subsequence such that $\lim\limits_{k\to +\infty}\mu_{n_k} = \liminf\limits_{n\to +\infty}\mu_n$. 
	Taking $u_{n_k}$ as a test function for \eqref{eq:De} with $\mu=\mu_{n_k}$, we get 
	\begin{equation*}\label{eq:test1}
	\mu_{n_k} \intO v_{n_k}^q \, dx - \intO |\nabla v_{n_k}|^q \, dx 
	= 
	\|\nabla u_{n_k}\|_p^{p-q} \left(1- \lambda_1(p) \int_\Omega v_{n_k}^p \, dx \right) \geq 0,
	\end{equation*}
	and hence the convergence of $v_{n_k}$ to $\varphi_p$ along a sub-subsequence yields
	$$
	\liminf_{n\to +\infty}\mu_n\ge 
	\frac{\int_\Omega |\nabla \varphi_p|^q \, dx}{\int_\Omega \varphi_p^q \, dx}=\beta_*. 
	$$
	Second, we choose a subsequence $\{\mu_{n_k}\}$ such that $\lim\limits_{k\to +\infty}\mu_{n_k} = \limsup\limits_{n\to +\infty}\mu_n$ and denote it, for simplicity, as $\{\mu_{k}\}$.	
	Using Picone's inequalities \eqref{eq:picone0} and \eqref{eq:piconeIlx} with $v=\varphi_p$, we get from \eqref{eq:De} with $\mu=\mu_k$ that 
	\begin{align*}
	\lambda_1(p) \intO \varphi_p^p \, dx &+ \mu_k \intO u_k^{q-p} \varphi_p^p \, dx \\
	&=
	\intO |\nabla u_k|^{p-2} \nabla u_k \nabla \left(\frac{\varphi_p^{p}}{u_k^{p-1}}\right)  dx
	+
	\intO |\nabla u_k|^{q-2} \nabla u_k \nabla \left(\frac{\varphi_p^{p}}{u_k^{p-1}}\right) dx \\
	&\leq 
	\intO |\nabla \varphi_p|^p \, dx 
	+ 
	\intO |\nabla \varphi_p|^{q-2} \nabla \varphi_p \nabla \left(\frac{\varphi_p^{p-q+1}}{u_k^{p-q}}\right) dx,
	\end{align*}
	which implies that
	\begin{equation}\label{eq:need10}
	\mu_k \intO u_k^{q-p} \varphi_p^p \, dx \leq \intO |\nabla \varphi_p|^{q-2} \nabla \varphi_p \nabla \left(\frac{\varphi_p^{p-q+1}}{u_k^{p-q}}\right) dx.
	\end{equation}
	Notice that both sides of \eqref{eq:need10} share the same homogeneity with respect to $u_k$, and hence we can replace $u_k$ by the normalized function $v_k$:
	\begin{equation}\label{eq:need1}
	\mu_k \intO v_k^{q-p} \varphi_p^p \, dx \leq \intO |\nabla \varphi_p|^{q-2} \nabla \varphi_p \nabla \left(\frac{\varphi_p^{p-q+1}}{v_k^{p-q}}\right) dx.
	\end{equation}
	The convergence $v_k \to \varphi_p$ in $C_0^1(\overline{\Omega})$ along a sub-subsequence yields the existence of 
	a constant $C>0$ such that $\varphi_p(x)\le C v_k(x)$ for all $x\in\Omega$ and all sufficiently large $k \in \mathbb{N}$.
	Therefore, since $0<\frac{\varphi_p}{v_k} \le C$ in $\Omega$, and $\frac{\varphi_p}{v_k} \to 1$ pointwise in $\Omega$, 
	the Lebesgue dominated convergence theorem guarantees 
	\begin{equation*}\label{eq:1exist2}
	\intO v_k^{q-p}\varphi_p^p\,dx \to \intO \varphi_p^q\,dx 
	\quad \text{and} \quad 
	\intO |\nabla \varphi_p|^{q-2}\nabla \varphi_p\nabla\left(\dfrac{\varphi_p^{p-q+1}}{v_k^{p-q}}\right)\,dx\to 
	\intO |\nabla \varphi_p|^q\,dx. 
	\end{equation*}
	Here, the latter convergence can be easily seen from the expansion
	\begin{align*}
	&\intO |\nabla \varphi_p|^{q-2} \nabla \varphi_p \nabla \left(\frac{\varphi_p^{p-q+1}}{v_k^{p-q}}\right) dx
	\\
	&=
	(p-q+1) \intO |\nabla \varphi_p|^q \, 
	\left(\frac{\varphi_p}{v_k}\right)^{p-q} \, dx
	-
	(p-q) \intO |\nabla \varphi_p|^{q-2} \nabla \varphi_p \nabla v_k \, 
	\left(\frac{\varphi_p}{v_k}\right)^{p-q+1} \, dx. 
	\end{align*}
	Thus, letting $k \to +\infty$ in \eqref{eq:need1}, we obtain 
	$$
	\limsup_{n\to +\infty}\mu_n=\lim_{k\to +\infty}\mu_k\le 
	\frac{\int_\Omega |\nabla \varphi_p|^q \, dx}{\int_\Omega \varphi_p^q \, dx}=\beta_*. 
	$$
	Consequently, the proof of the assertion \ref{prop:behaviour:1} is complete.
	
	\medskip
	\ref{prop:behaviour:2}
	Let $\|\nabla u_n\|_p \to 0$ as $n \to +\infty$.
	This implies that $u_n \to 0$ a.e.\ in $\Omega$. Consequently, by Egorov's theorem, $u_n \to 0$ uniformly on some subset of $\Omega$ of positive measure, which yields
	$$
	\intO u_n^{q-p}\varphi_q^p\,dx \to +\infty
	\quad \text{as}~ n \to +\infty.
	$$
	Thus, using the Picone inequalities \eqref{eq:picone0} and \eqref{eq:piconeIlx} as in \eqref{eq:pic1115}, we get \eqref{eq:1}:
	$$
	(\mu_n-\lambda_1(q))\intO u_n^{q-p}\varphi_q^p\,dx 
	\le \intO |\nabla \varphi_q|^p\,dx-\lambda_1(p)\intO \varphi_q^p\,dx < +\infty,
	$$
	and therefore $\limsup\limits_{n\to +\infty}\mu_n\le \lambda_1(q)$. 
	Recalling now that $\liminf\limits_{n\to +\infty}\mu_n \ge \lambda_1(q)$, we finish the proof of the assertion \ref{prop:behaviour:2}.
\end{proof} 

\begin{proof}[Proof of Theorem \ref{thm:summary}] 
	First, we recall that \eqref{eq:De} has a positive solution if $\mu\in (\lambda_1(q),\beta_*)$, see \cite[Theorem 2.5 (i)]{BobkovTanaka2017}. 
	Therefore, $\widetilde{\mu}\ge \beta_*$. 
	Let $\{\mu_n\}$ be a sequence convergent to $\widetilde{\mu}$ such that \eqref{eq:De} with $\mu=\mu_n$ has a positive solution $u_n$. 
	Fixing any $v \in C_0^1(\overline{\Omega})$ and applying Picone's inequalities \eqref{eq:piconeB} and \eqref{eq:picone0}, we get from \eqref{eq:De} with $\mu=\mu_n$ that
	\begin{align}
	\notag
	\lambda_1(p) \intO u_n^{p-q} v^q \, dx & + \mu_n \intO v^q \, dx \\
	\notag
	&=
	\intO |\nabla u_n|^{p-2} \nabla u_n \nabla \left(\frac{v^{q}}{u_n^{q-1}}\right)  dx
	+
	\intO |\nabla u_n|^{q-2} \nabla u_n \nabla \left(\frac{v^{q}}{u_n^{q-1}}\right) dx \\
	\label{eq:thm:summary:1}
	&\leq 
	\frac{q}{p} \intO |\nabla v|^{p} \, dx + \frac{p-q}{p}\intO|\nabla u_n|^{p} \, dx 
	+ 
	\intO |\nabla v|^{q} \, dx.
	\end{align}
	If we suppose that $\widetilde{\mu}=+\infty$, then Proposition \ref{prop:behavior} \ref{prop:behaviour:1} guarantees the boundedness of $\|\nabla u_n\|_p$, whence we get a contradiction to \eqref{eq:thm:summary:1}. Thus, $\widetilde{\mu}<+\infty$. 
	
	By the definition of $\widetilde{\mu}$, \eqref{eq:De} has no positive solution for any $\mu>\widetilde{\mu}$. Moreover, there is no nonzero solution if $\mu \leq \lambda_1(q)$, see \cite[Proposition 1]{BobkovTanaka2015} and \cite[Proposition 13]{BobkovTanaka2017}. Furthermore, \cite[Theorem 2.2 (i)]{BobkovTanaka2015} implies that \eqref{eq:De} has at least one positive solution for any $\mu \in (\lambda_1(q),\widetilde{\mu})$.
	
	It remains to prove that if $\widetilde{\mu}>\beta_*$, then \eqref{eq:De} with $\mu=\widetilde{\mu}$ possesses a positive solution. 
	Choose a sequence $\{\mu_n\}$ such that $\mu_n \to \widetilde{\mu}$ and \eqref{eq:De} with $\mu=\mu_n$ has a positive solution $u_n$. 
	Proposition \ref{prop:behavior} \ref{prop:behaviour:1} guarantees the boundedness of 
	$\|\nabla u_n\|_p$. 
	So, we may assume, by passing to a subsequence, that $u_n$ converges to some nonnegative function $u_0$ weakly in $\W$ 
	and strongly in $L^p(\Omega)$. 
	Letting $n\to +\infty$ in 
	\begin{align*}
	&\intO |\nabla u_n|^{p-2}\nabla u_n\nabla(u_n-u_0) \,dx 
	+ \intO |\nabla u_n|^{q-2}\nabla u_n\nabla(u_n-u_0) \,dx \\
	&= \lambda_1(p) \intO u_n^{p-1} (u_n-u_0) \, dx 
	+ \mu_n\intO u_n^{q-1} (u_n-u_0) \,dx,
	\end{align*}
	we get 
	$$
	\lim_{n\to +\infty}
	\left(\intO |\nabla u_n|^{p-2}\nabla u_n(\nabla u_n-\nabla u_0) \,dx
	+\intO |\nabla u_n|^{q-2}\nabla u_n(\nabla u_n-\nabla u_0) \,dx
	\right)=0, 
	$$
	which yields the strong convergence $u_n \to u_0$ in $\W$, see \cite[Remark 3.5]{BobkovTanaka2017}. 
	At the same time, since $\widetilde{\mu}>\beta_*>\lambda_1(q)$, Proposition \ref{prop:behavior} \ref{prop:behaviour:2} gives $\|\nabla u_0\|_p=\lim\limits_{n\to +\infty}\|\nabla u_n\|_p>0$, and hence $u_0$ is nonzero. 
	Thus, $u_0$ is a positive solution of \eqref{eq:De} with $\mu=\widetilde{\mu}$.
\end{proof}

\section*{Acknowledgements}
V.~Bobkov was supported by the project LO1506 of the Czech Ministry of Education, Youth and Sports, and by the grant 18-03253S of the Grant Agency of the Czech Republic.
M.~Tanaka was supported by JSPS KAKENHI Grant Number JP 19K03591. 
The authors would like thank the anonymous reviewers whose comments and suggestions helped to improve and clarify the manuscript.


\end{document}